\documentclass[11pt, leqno]{amsart}
\makeatletter  
\usepackage{amsfonts,delarray,amssymb,amsmath,amsthm,a4,a4wide}
\usepackage{color}

\theoremstyle{plain}
\newtheorem{theorem}{Theorem}[section]
\newtheorem{corollary}[theorem]{Corollary}
\newtheorem{lemma}[theorem]{Lemma}
\newtheorem{proposition}[theorem]{Proposition}
\newtheorem{definition}[theorem]{Definition}

\numberwithin{equation}{section}
\usepackage[margin=1in]{geometry}

\newcommand{\R}{\mathbb R}
\newcommand{\be}{\begin{equation}}
\newcommand{\ee}{\end{equation}}

\newcommand{\ep}{\eps}

\newcommand{\eps}{\varepsilon}
\newcommand{\ph}{\varphi}
\newcommand{\ov}{\overline}
\newcommand{\Om}{\Omega}
\newcommand{\p}{\partial}

\newcommand{\comment}[1]{}

\newenvironment{myindentpar}[1]%
{\begin{list}{}%
         {\setlength{\leftmargin}{#1}}%
         \item[]%
}
{\end{list}}

\title[Schauder estimates and smoothness of eigenfunctions ]{
Schauder estimates for degenerate Monge-Amp\`ere equations and smoothness of the eigenfunctions
}
\author{Nam Q. Le}
\address{Department of Mathematics, Indiana University, Bloomington, IN 47405, USA}
\email{nqle@indiana.edu}
\author{Ovidiu Savin}
\address{Department of Mathematics, Columbia University, New York, NY 10027, USA}
\email{savin@math.columbia.edu}
\thanks{ The first author was supported by NSF grant DMS-1500400. The second author was supported by NSF grant DMS-1200701.}

\begin{document}

\begin{abstract}
We obtain $C^{2,\beta}$ estimates up to the boundary for solutions to degenerate Monge-Amp\`ere equations of the type
$$
\det D^2 u = f~~\text{in}~\Omega, \quad \quad ~f\sim \text{dist}^{\alpha}(\cdot, \p\Omega)~\text{near}~\p\Omega,~\alpha>0.
$$
As a consequence we obtain global $C^\infty$ estimates up to the boundary for the eigenfunctions of the Monge-Amp\`ere operator $(\det D^2 u)^{1/n}$ on smooth, bounded, uniformly 
convex domains in $\R^n$.
\end{abstract}
\maketitle

\section{Introduction}

In this paper we develop Schauder estimates up to the boundary for degenerate Monge-Amp\`ere equations of the type
\begin{equation}
\det D^2 u = f~~\text{in}~\Omega, \quad \quad ~f\sim d_{\p\Omega}^{\alpha}~\text{near}~\p\Omega,
\label{MAxn}\end{equation}
and apply them to prove global smoothness for the eigenfunctions of the Monge-Amp\`ere operator $(\det D^2 u)^{1/n}$. Throughout the paper, $d_{\p\Omega}$ represents the 
distance to the boundary of the domain $\Omega$, and $\alpha>0$ is a positive power. 

Boundary estimates for the Monge-Amp\`ere equation in the nondegenerate case where $f\in C(\ov \Om)$ and $f>0$ were obtained starting with 
the works of Ivo\u{c}kina \cite{I}, Krylov \cite{K}, Caffarelli-Nirenberg-Spruck \cite{CNS} (see also \cite{W}). Also in the nondegenerate case, global $C^{2,\alpha}$ estimates under sharp conditions
on the right hand side and boundary data were obtained recently by Trudinger-Wang \cite{TW1} and the second author \cite{S1}.

In \cite{S2}, the second author established a boundary localization theorem and then applied it to obtain
$C^2$ estimates at the boundary for solutions of (\ref{MAxn})
under natural conditions on the boundary data and the right hand side. In the present work we investigate further regularity of such solutions. We use the boundary $C^2$ estimates in \cite{S2} and 
perturbation arguments \cite{C, CC} to establish basic boundary H\"older second derivative estimates for solutions to (\ref{MAxn}) when the boundary $\p\Omega$, boundary data of $u$ on $\p\Omega$ and $f$ are more regular; see 
Theorems \ref{T01} and \ref{T02}. Our results can be viewed as the degenerate counterparts of the works cited above. They easily give global $C^{2,\beta}$ estimates up to the boundary for the eigenfunctions of the Monge-Amp\`ere operator $(\det D^2 u)^{1/n}$; 
see Theorem \ref{T03}. This is the key step in settling the open problem on global smoothness of the Monge-Amp\`ere eigenfunctions in all dimensions; see Theorem \ref{T035}.

Before stating our main results we recall the notion for a function to be $C^{2,\beta}$ at a point (see \cite{CC}). We say that $u$ is $C^{2, \beta}$ at $x_0$ if there exists a quadratic polynomial $Q_{x_0}$ such that, in the domain of definition of $u$, $$u(x)=Q_{x_0}(x) + O(|x-x_0|^{2+\beta}).$$
 Throughout this paper we refer to a linear map $A$ of the form
 $$Ax=x+\tau x_n, \quad \quad \mbox{with} \quad \tau \cdot e_n=0,$$
as a {\it sliding along} $x_n=0$. Notice that the map $A$ is the identity map when it is restricted to $x_n=0$ and it becomes a translation of vector $s \tau$ when it is restricted to $x_n=s$.

Let $\Omega$ be a bounded convex domain in the upper-half space such that $\p \Om$ is $C^{1,1}$ at the origin, that is $0 \in \p \Om$ and
\begin{equation}\label{01}
\Om \subset \{x_n>0\},\quad \mbox{ and $\Om$ has an interior tangent ball at the origin.}
\end{equation}
We are interested in the behavior near the origin of a convex solution $u\in C(\ov \Om)$ to the equation
\begin{equation}\label{02}
\det D^2 u=g(x) \, d_{\p \Om}^\alpha, \quad \quad \quad \alpha>0,
\end{equation}
where $g$ is a nonnegative function that is continuous at the origin and $g(0)>0$.

We state below our main Schauder estimates for solutions to \eqref{02}. In Section \ref{s2} we will give more precise quantitative versions of these theorems.

Our first result is the following pointwise $C^{2,\beta}$ estimates at the boundary. 
\begin{theorem}\label{T01}
Let $\Om$, $u$ satisfy \eqref{01}, \eqref{02} above and let $\beta \in (0,\frac{2}{2+\alpha})$. Assume that
$$ u(0)=0, \quad \nabla u(0)=0, \quad u=\ph \quad \mbox{on $\p \Om$},$$
and the boundary data $\ph$ is $C^{2,\beta}$ at $0$, and it separates quadratically away from $0$.

If $g \in C^\gamma$ at the origin where $\gamma=\beta \, \, \frac{2+\alpha}{2}$ then $u$ is $C^{2,\beta}$ at $0$. Precisely, there exist a sliding $A$ along $x_n=0$ and a constant $a>0$ such that
$$u(Ax)=Q_0(x')+ a \, \,  x_n^{2+\alpha} + O(|x'|^2+x_n^{2+\alpha})^{1+\frac \beta 2},$$
where $Q_0$ represents the quadratic part of the boundary data $\ph$ at the origin.
\end{theorem}

In fact the conclusion above states that $u$ is pointwise $C^{2,\beta}$ with respect to the degenerate metric specific to our equation (\ref{02}). This is a stronger statement than the pointwise $C^{2,\beta}$ with respect to the Euclidean metric.

In the second theorem we establish the $C^{2,\beta}$ regularity of the solutions to \eqref{02} in a neighborhood of the origin.
\begin{theorem}\label{T02}
Let $\beta \in (0,\frac{2}{2+\alpha})$, and $\beta \le \alpha$. Suppose that $u$ satisfies the equation \eqref{02} near the origin with
$$\p \Om \in C^{2,\beta}, \quad u|_{\p \Om} \in C^{2,\beta}, \quad g \in C^\gamma \quad \mbox{where} \quad \gamma=\beta \, \, \frac{2+\alpha}{2}.$$
If $u$ separates quadratically from its tangent plane at $0$ then
$$u \in C^{2,\beta}(\overline \Om \cap B_\delta(0)) \quad \mbox{for some small $\delta>0$}.$$
\end{theorem}

We remark that the hypothesis $\beta \le \alpha$ is necessary in the second theorem. This can be easily seen from the expansion given in Theorem \ref{T01}.

As a consequence of our Schauder estimates, we obtain the global $C^{2,\beta}$ regularity, for all $\beta < \frac{2}{2+n}$,  of the eigenfunctions to the Monge-Amp\`ere operator $(\det D^2 u)^{1/n}$ on smooth, bounded and uniformly convex domains in $\R^n$. This problem was 
first investigated by Lions \cite{L}. He showed that there exists a unique (up to positive multiplicative constants) nonzero convex eigenfunction $u\in C^{1,1}(\overline{\Omega})\cap C^{\infty}(\Omega)$ to the eigenvalue problem for the Monge-Amp\`ere operator $(\det D^2 u)^{1/n}$: 
\begin{equation}\label{EP}
\left\{
 \begin{alignedat}{2}
   (\det D^2 u)^{\frac{1}{n}} ~& = \lambda|u|~&&\text{in} ~  \Omega, \\\
u &= 0~&&\text{on}~ \p\Omega.
 \end{alignedat} 
  \right.
  \end{equation}
We note that $\lambda $, called the Monge-Amp\`ere eigenvalue of $\Om$,  has an interesting stochastic interpretation given by Lions \cite{L}. A 
variational characterization of $\lambda$ was found by Tso \cite{T}.

The question of global higher derivative estimates up to the boundary of the eigenfunction $u$ is a well known open problem, see for example Trudinger and Wang's survey paper \cite{TW2}. 

In the two dimensional case, Hong-Huang-Wang in \cite{HHW} resolved this question in the affirmative. In
higher dimensions, the global $C^2 (\ov \Om)$ estimate of the eigenfunction $u$ was obtained by the second
author in \cite{S2}. In Theorem \ref{T035} below we settle the question of global smoothness of the eigenfunction $u$ in higher dimensions.

Notice that $u \in C^{1,1}(\overline{\Omega})$ implies that we can write $|u|=g \, d_{\p \Om} $ with $g \in C^{0,1}(\overline{\Omega})$. Since $u=0$ on $\p \Om$ we also have 
that $u$ separates quadratically on $\p \Om$ from its tangent planes, and therefore we can apply Theorem \ref{T02}.

\begin{theorem}\label{T03} Let $\Om$ be a bounded and uniformly convex domain in $\R^n$.
Assume $\p \Om \in C^{2,\beta}$ with $\beta \in (0,\frac{2}{2+n})$, and $u$ satisfies \eqref{EP}. 
 Then $u \in C^{2,\beta}(\ov \Om)$.
\end{theorem}

Our final theorem shows that the Monge-Amp\`ere eigenfunctions are in fact smooth up to the boundary of $\Omega$ if $\p\Omega$ is smooth. 
To do this, we 
first perform a Hodograph transform and then a partial Legendre transform to an eigenfunction $u$.  
The linearized operator of the resulting equation can be treated as a degenerate Grushin-type operator. By using the result of Theorem \ref{T03} and Schauder estimates for Grushin-type operators, 
we obtain our global smoothness result.

\begin{theorem}\label{T035} Let $\Om$ be a bounded and uniformly convex domain in $\R^n$.
Assume $\p \Om \in C^{\infty}$ and $u$ satisfies  \eqref{EP}.  Then $u \in C^{\infty}(\ov \Om)$.
\end{theorem}

The paper is organized as follows. In Section \ref{s2} we introduce notation and state the quantitative versions of our Schauder estimates. By suitable rescalings we reduce Theorems \ref{T01} and \ref{T02} to the Propositions \ref{c2b} and \ref{C2b}. These propositions are concerned with the $C^{2,\beta}$ estimates for solutions
of \eqref{02} which are small perturbations of the particular solution $U_0$ introduced in Section \ref{s2}.
The proofs of Propositions \ref{c2b} and \ref{C2b} will be given in Sections \ref{s3} and \ref{s4} respectively. In Section \ref{s5}, we prove Theorem \ref{T035}. Some technical results used in the 
proofs of the main theorems are collected in the final section, Section \ref{res_sec}.

\section{Quantitative versions of the Schauder estimates}
\label{s2}
We introduce some notation. We denote points in $\R^n$ ($n\geq 2$) as
$$x=(x_1,...,x_n)=(x',x_n), \quad \quad \quad x' \in \R^{n-1}.$$
Let $\R^n_{+}=\{x\in\R^n: x_n>0\}$ be the upper half-space. Denote by $I$ and $I'$ the identity matrices of size $n$ and $n-1$ respectively.
We denote by $B_r(x)$ the ball of radius $r$ and center $x$, and by $B_r'(x')$ the ball in $\R^{n-1}$ of radius $r$ and center $x'$. For simplicity, we use $B_r$ to denote $B_r(0)$
and $B_r^{+}:= B_r \cap \{x_n\geq 0\}.$ 

Given a convex function $u$ defined on a convex set $\ov \Om$, we denote by $S_h(x_0)$ the section centered at $x_0$ and height $h>0$,
$$S_h(x_0):=\{x \in \ov \Om | \quad  u(x) < u(x_0)+\nabla u(x_0) \cdot (x-x_0) +h \}.$$
We denote for simplicity $S_h=S_h(0)$, and sometimes when specifying the dependence on the function $u$ we use the notation $S_h(u)=S_h$.

Most of the time we write our estimates in the sections of the particular function
$$U_0(x):=\frac 12 |x'|^2 + \frac{1}{(1+\alpha)(2+\alpha)}\, x_n^{2+\alpha},\quad x\in\R^n_{+}.$$
We denote by $A$ a sliding along $x_n=0$, i.e., $Ax=x+\tau x_n$ for some $\tau\in \R^n$ with $\tau \cdot e_n=0$. Also we denote by $F_h$ the diagonal matrix
$$F_h=\text{diag}(h^\frac 12, h^\frac 12,..,h^\frac 12, h ^\frac{1}{2+\alpha}).$$
We often rescale solutions in $S_h(u)$ according to the formula
$$u_h(x)=h^{-1}u(A F_hx).$$
For a $C^{\gamma}$ continuous function $f$ on a set $S$, we denote by $[f]_{C^{\gamma}(S)}$ its homogeneous seminorm, i.e,
$$[f]_{C^{\gamma}(S)} =\sup_{x, y\in S,~x\neq y}\frac{|f(x)-f(y)|}{|x-y|^{\gamma}}.$$
When $\gamma=1$, the right hand side is denoted by $[f]_{C^{0, 1}(S)}$.

Throughout the paper we think of the constants $n$, $\alpha$ and $\beta$ as being fixed. We refer to all positive constants depending only $n$, $\alpha$ and $\beta$ 
as {\it universal constants} and we denote them by $c$, $C$, $c_i$, $C_i$. 
The dependence of various constants also on other parameters like $\rho, \rho'$ will be denoted by $C(\rho, \rho'), c(\rho, \rho')$.

We start by writing quantitative versions of Theorems \ref{T01} and \ref{T02}.
\begin{theorem}\label{T2.1}
Let $\alpha, \beta$, $\gamma$ be as in Theorem \ref{T01} and $\rho$, $\rho' \ll \rho$ be small parameters. Let $\Om$ be a bounded convex domain in the upper half-space with $0\in\p\Om$, and assume $\Om$ has an interior ball of radius $\rho$ at the origin. 
Let $u\in C(\ov \Om)$ be convex such that
$u(0)=0, \nabla u(0)=0$
with $$ \left |u(x)- \frac 12 |x'|^2 \right | \leq \frac{1}{\rho'}|x'|^{2+\beta} ~\text{in}~\p\Om\cap B_{\rho},$$
and $u \ge \rho'$ on $\p \Om \setminus B_\rho$.
Assume $$ \det D^2 u= g \, d_{\p \Om}^\alpha \quad \mbox{in} \quad \Om \cap B_\rho, \quad \quad \det D^2 u \le \frac{1}{\rho'} \quad \mbox{in} \quad \Om \setminus B_\rho $$
with $$|g-1|\le \frac{1}{\rho'}|x|^{\gamma}~\text{in}~\Om\cap B_{\rho}.$$
Then there exists a sliding $A$ along $x_n=0$  with $|A-I|\leq C$ such that
$$\left| u(Ax) - U_0(x)\right|\leq C(|x'|^2 + x_n^{2+\alpha})^{1+\frac{\beta}{2}}~\text{in}~ B_{\rho},$$ with $C$ depending on $n$, $\alpha$, $\beta$, $\rho$, $\rho'$.
\end{theorem}

We remark that in the proof of Theorem \ref{T2.1} we only use the following property of the distance function
$$x_n \ge d_{\p \Om}(x) \ge \left(x_n- \rho^{-1}|x'|^2\right)^+ .$$

\begin{theorem}\label{T2.2} 
Let $\alpha, \beta$, $\gamma$ be as in Theorem \ref{T02} and $\rho$, $\rho' \ll \rho$ be small parameters. Assume $\p \Omega \in C^{2,\beta}$ 
in $B_\rho$ with $\|\p \Om\|_{C^{2,\beta}} \le 1/\rho'$. Let $u\in C(\ov \Om)$ be convex such that
$$u(0)=0, \quad \nabla u(0)=0, \quad u=\ph(x') \quad \mbox{on $\p \Om \cap B_\rho$,}$$
with $$\|\ph\|_{ C^{2,\beta}(B_\rho')} \le \frac {1}{\rho'} , \quad \rho' I' \le D_{x'}^2 \ph(0) \le \frac{1}{\rho'} \, I',$$ and $u \ge \rho'$ on $\p \Om \setminus B_\rho$.
Assume $$ \det D^2 u= g \, d_{\p \Om}^\alpha \quad \mbox{in} \quad \Om \cap B_\rho, \quad \quad \det D^2 u \le \frac{1}{\rho'} \quad \mbox{in} \quad \Om \setminus B_\rho $$
with $$g \in C^\gamma(\ov \Om \cap B_\rho), \quad \quad \|g\|_{C^\gamma} \le \frac{1}{\rho'}, \quad \rho'\le g(0).$$
Then $$\|u\|_{ C^{2,\beta} (\ov \Om \cap B_\delta)} \le C,$$ with $\delta$ and $C$ depending on $n$, $\alpha$, $\beta$, $\rho$, $\rho'$.
\end{theorem}

As mentioned in the introduction, the pointwise $C^2$ estimates for solutions to \eqref{02} were established in \cite[Theorem 2.4]{S2}. This means that after a suitable 
rescaling (see Lemma \ref{good_rescale}), it suffices to prove our quantitative theorems in the setting when $u$, and the data $g$, $\ph$, $\p \Om$ are small perturbations of the particular 
solution $U_0$,
$$U_0(x):=\frac 12 |x'|^2 + \frac{1}{(1+\alpha)(2+\alpha)}\, x_n^{2+\alpha},$$
with the data $g \equiv 1$, $ \ph =\frac 12 |x'|^2$, $\p \Om=\{x_n=0\}$. We state below our Theorems \ref{T2.1} and \ref{T2.2} in this setting.

Let $\Om$ be a convex set with boundary given by the graph of a $C^{1,1}$ convex function $q(x')$, i.e.,
$$\Om=\{x_n > q(x')\}, \quad \quad q(0)=0, \quad \nabla_{x'}q(0)=0.$$
Let $u$ be a convex function defined in $\ov \Om \cap S_1(U_0)$ that satisfies
\begin{equation}\label{Ma1}
\det D^2u=g \, d_{\p \Om}^\alpha \quad \mbox{in} \quad \Om \cap S_1(U_0), 
\end{equation} 
\begin{equation}\label{Ma2}
u(0)=0, \quad \nabla u(0)=0.
\end{equation}
and
 \begin{equation}\label{Ma3}
u(x)=\ph(x') \quad \mbox{on} \quad \p \Om.
\end{equation}

First, we state the pointwise $C^{2,\beta}$ estimates at the origin for solutions to \eqref{Ma1} which are perturbations of $U_0$.
\begin{proposition}[Pointwise $C^{2,\beta}$ estimates]\label{c2b}
Assume $u$ satisfies \eqref{Ma1}, \eqref{Ma2}, \eqref{Ma3} with
$$|u-U_0| \le \eps, \quad |g-1|\leq \ep |x|^\gamma \quad \quad \quad \mbox{in} \quad \Omega \cap S_1(U_0),$$
with $\gamma= \frac{2+\alpha}{2}\, \,  \beta\, \in (0,1)$, and
 $$|x_n|\leq \eps |x'|^2 \quad \mbox{and} \quad \left|\ph-\frac 12 |x'|^2 \right |\leq \eps |x'|^{2+\beta} \quad \quad \mbox{on} \quad \p\Omega\cap S_1(U_0),$$ 
for some $\eps \le \eps_0$ universal. Then there exists a sliding $A$ with $|A-I| \le C \eps$ such that
$$|u(Ax)-U_0| \le C \eps (|x'|^2+x_n^{2+\alpha})^{1+\frac \beta 2}.$$
\end{proposition}
Next, we state the $C^{2,\beta}$ estimates in a neighborhood of the origin for solutions to \eqref{Ma1} which are perturbations of $U_0$.
\begin{proposition}[$C^{2,\beta}$ estimates]\label{C2b}
Assume $u$ satisfies \eqref{Ma1} and \eqref{Ma3} with
$$|u-U_0| \le \eps_0, \quad [g]_{C^\gamma}\le \eps_0, \quad [D^2 \ph]_{C^\beta} \le \eps_0, \quad  \|q\|_{C^{2, \beta}}\le \eps_0,$$
and $\beta \in (0,\frac{2}{2+\alpha})$, $\beta \le \alpha$, $\gamma= \frac{2+\alpha}{2}\, \,  \beta$, and $\eps_0$ small depending on $n$, $\beta$, $\alpha$. 
Then $$ \|u\|_{C^{2,\beta}(\ov \Om \cap S_{1/2}(U_0))} \le C.$$ 
\end{proposition}

 As mentioned above, to reduce the proof of Theorem \ref{T2.2} (or Theorem \ref{T2.1}) to that of Proposition \ref{C2b} (or Proposition \ref{c2b}),
we need to show that by suitably rescaling $u$ and the domain $\Omega$, the new function is well approximated by $U_0$ on a new domain with very flat boundary near the origin. More 
precisely, we have the following rescaling lemma whose proof we postpone till Section \ref{res_sec}.
We recall that we denote by $F_h$ the diagonal matrix
$$F_h=\text{diag}(h^\frac 12, h^\frac 12,..,h^\frac 12, h ^\frac{1}{2+\alpha}).$$

\begin{lemma} [Rescaling]
\label{good_rescale}
Suppose that $u$ satisfies the hypotheses of Theorem \ref{T2.2}. Then there exist two linear transformations $D$ symmetric, and $A$ a sliding, with $|D|,|A|\leq C$ such that the rescaling  $(u_h, \Omega_h)$ of  $(u, \Omega)$ given by
$$u_h (x) =\frac{u(DAF_h x)}{h}, \quad \quad ~x\in \Omega_h = (DAF_h)^{-1}\Omega,$$
and the boundary data $\varphi_h, q_h$ of $(u_h,\Om_h)$ satisfy in $S_1 (U_0)\cap \Omega_h$ the hypothesis of Proposition \ref{C2b}, provided that $h=h(\ep_0)$ is chosen small depending on the universal constants and $\ep_0$.
\end{lemma}

We now say a few words about the proofs of the above propositions. 

We prove Propositions \ref{C2b} by combining the interior $C^{2,\beta}$ estimates together with pointwise
$C^{2,\beta}$ estimates in Proposition \ref{c2b} applied at points on the boundary. The geometric properties of sections of $u$ near the boundary play a crucial role.

The key ingredient in the proof of Proposition \ref{c2b} is the following iteration lemma.

\begin{lemma}[Iteration step]\label{il}
 Assume that for some small $\ep \le \ep_0$, u satisfies \eqref{Ma2}, \eqref{Ma3},
 $$|u-U_0|\leq \ep \quad \quad \mbox{in} \quad \Omega \cap S_1(U_0),$$
and in this set
$$(1+\delta_0 \ep) x_n^\alpha \ge \det D^2u \ge (1-\delta_0 \ep)\left((x_n-\delta_0 \ep)^+ \right)^\alpha.$$
If $$|x_n|\leq \delta_0 \eps |x'|^2 \quad \mbox{and} \quad \left|\ph-\frac 12 |x'|^2 \right |\leq \delta_0\eps |x'|^2 \quad \quad \mbox{on} \quad \p\Omega\cap S_1(U_0),$$ 
then there exists a rescaling $\tilde u$ of $u$ at height $h=\theta_0$,
$$\tilde u(x):=\frac 1 h u(A_h F_h x), \quad x\in \tilde{\Omega}:=(A_h F_h)^{-1}\Omega \quad \mbox{with} \quad  h=\theta_0,   \quad |A_h -I| \le C \ep,$$
such that 
 $$|\tilde u-U_0|\leq \ep \theta_0^\frac \beta 2 \quad \quad \mbox{in} \quad \tilde{\Omega} \cap S_1(U_0).$$
The constants above $\eps_0$, $\delta_0$, $\theta_0$ (small) and $C$ are universal: they depend only on $n$, $\alpha$ and $\beta$.
\end{lemma}

The proof of Lemma \ref{il} follows the standard perturbation method: we approximate $u-U_0$ by solutions of the Grushin-type operator
$$ L \, w=0, \quad \quad L \, w:=x_n^{\alpha}\Delta_{x^{'}} w + w_{nn},$$
which is the linearized equation of the Monge-Amp\`ere equation $$\det D^2 U_0=x_n^\alpha.$$
Before we proceed with the proof we state the $C^{2, \frac{2}{2+\alpha}}$ estimates for solutions of such linearized equation.

\begin{lemma}
\label{lineq}
 Assume that for some $\alpha>0$, $w$ solves
\begin{equation}L \, w:=x_n^{\alpha}\Delta_{x'} w + w_{nn}=0~\text{in}~ B_1^{+},~~w=0~ \text{on}~ \{x_n=0\}
\label{w-eq}\end{equation}
and $\|w\|_{L^{\infty}(B_{1}^{+})}\leq 1.$
Then in $B_{1/2}^{+}$, $w$ satisfies
$$\left|w(x) -  \left (a_0 + a'\cdot x'\right)x_{n}\right| \le C_{0} \left (|x^{'}|^2 + x_{n}^{2+\alpha}\right)^\frac{3+\alpha}{2+\alpha}$$
with $C_0$ universal, and $a_0\in \R, a'\in \R^{n-1}$ with $|a_0|$, $|a'|$ bounded by a universal constant (i.e. depending only on $n$ and $\alpha$).
\end{lemma}
 
The proof of Lemma \ref{lineq} is standard and will be given in Section \ref{res_sec}.

\section{Proof of Proposition \ref{c2b}}
\label{s3}
In this section, we prove Proposition \ref{c2b}.
We start with the proof of the main lemma.

\begin{proof}[Proof of Lemma \ref{il}]

Assume $u$ satisfies the hypotheses of the lemma and denote
$$v:=\frac{u-U_0}{\ep}.$$
 
 {\it Step 1:} We show that $v$ is well approximated by solutions to the linearized equation \eqref{w-eq}. Precisely, for any small $\eta>0$ we can find a solution $w$ to 
 $$L \, w=0, \quad ~~w=0~ \text{on}~ \{x_n=0\}$$ such that
 $$|v-w| \le \eta \quad \quad \mbox{in} \quad S_{1/2}(U_0) \cap \ov \Om,$$
 provided that $\eps_0(\eta)$, $\delta_0(\eta)$ are chosen sufficiently small depending on $\eta$ and the universal constants.
 
 \
 
 In other words, we need to show that for a sequence of $\eps_0$, $\delta_0 \to 0$ the corresponding $v$'s converge uniformly  in $S_{1/2}(U_0)\cap\ov\Om$ to a solution of $Lw=0$ (along a subsequence).
 
First we show that $v$ separates almost ``linearly" away from $\{x_n=0\}$. 

For each $x_0 \in S_{1/2}(U_0) \cap \{x_n=0\}$ we consider the upper barrier for $u$:
$$ \bar u_{x_0}(x)= \frac{1}{2}|x^{'}|^2 +8 \ep |x^{'}-x_0^{'}|^2 +\frac{1-\delta_0 \ep}{(1 + 16\ep)^{n-1}}\frac{\left ((x_n-  \delta_0\ep)^{+}\right)^{2+\alpha}}{(1+\alpha)(2+\alpha)} 
+ C\ep x_n + 2\delta_0\ep,$$
for some $C$ large to be determined.

 Then, we have
\begin{equation}
   \det D^2 \bar u_{x_0}  = (1-\delta_0 \ep)\left((x_n- \delta_0\ep)^{+}\right)^{\alpha} \leq \det D^2 u~\text{in} ~ S_{1}(U_0)\cap \Omega,
   \label{detwbar}
   \end{equation}
and we easily find that on $\p S_1 (U_0)\cap \Omega$,
 $$\bar{u}_{x_0} \geq U_0 + 8 \ep |x'-x^{'}_0|^2 + \frac C 2 \ep x_n  \ge 1+ \ep \ge u,$$
 and on $\p \Om \cap S_1 (U_0)$,
 $$\bar{u}_{x_0} \geq  \frac{1}{2}|x^{'}|^2 + 2\delta_0 \ep \ge u.$$
 
 By the maximum principle we obtain
 $u \le \bar{u}_{x_0} $ in $\Om \cap S_1(U_0)$ which implies that
 $$v(x_0',x_n) \le 2\delta_0+ C x_n.$$
 
 The opposite inequality follows similarly by considering the lower barrier
 
$$ \underbar u_{x_0}(x)= \frac{1}{2}|x^{'}|^2 -8 \ep |x^{'}-x_0^{'}|^2 +\frac{1+\delta_0 \ep}{(1 - 16\ep)^{n-1}}\frac{x_n^{2+\alpha}}{(1+\alpha)(2+\alpha)} - C\ep x_n -2 \delta_0\ep.$$ 
 
 Then
$$ \det D^2 \underbar u  = (1+\delta_0 \ep) x_n^{\alpha} \ge \det D^2 u~\text{in} ~ S_{1}(U_0)\cap \Omega,$$
and on $\p S_1 (U_0)\cap \Omega$,
 $$\underbar{u}_{x_0} \leq U_0 - 8 \ep |x'-x^{'}_0|^2 - \frac C 2 \ep x_n  \le 1- \ep \le u,$$
 and on $\p \Om \cap S_1 (U_0)$,
 $$\underbar{u}_{x_0} \leq  \frac{1}{2}|x'|^2 - 2\delta_0 \ep \le u.$$
 
 By the maximum principle we obtain
 $\underbar{u}_{x_0}\leq u $ in $\Om \cap S_1(U_0)$ which implies that
 $$v(x_0',x_n) \ge -2\delta_0- C x_n.$$
 
 In conclusion 
 \begin{equation}\label{vl}
 -2\delta_0 -C x_n \le v \le 2\delta_0 + C x_n \quad \mbox{in} \quad S_{1/2}(U_0) \cap \ov \Om.
 \end{equation}
 
 This inequality provides the uniform convergence of the $v$'s near $\{x_n=0\}$. It remains to prove the uniform convergence on compact subsets of $S_1(U_0) \cap \{x_n >0 \}$.
 
 Fix a ball $B_{2 \rho}:=B_{2 \rho}(z_0) \subset S_1(U_0)$ and consider the solution $u_0$ to
 $$\det D^2 u_0 =x_n^ \alpha \quad \mbox{in} \quad B_{\frac 32 \rho}, \quad u_0=u \quad \mbox{on} \quad \p B_{\frac 32 \rho},$$ 
 and the corresponding $$v_0:= \frac{u_0-U_0}{\ep}.$$
 We claim that if $\rho\geq 2\eps\delta_0$ then
 \begin{equation}\label{uu0}
 |u-u_0| \le C(\rho) \ep \delta_0 \quad \mbox{in} \quad B_{\frac 32 \rho}(z_0).
 \end{equation}
 Indeed, notice that in this ball $ B_{\frac 32 \rho}(z_0)$, we have
 $$ |\det D^2u -\det D^2 u_0| \le C(\rho) \eps \delta_0, \quad \mbox{and} \quad \det D^2u, \, 
 \det D^2u_0 \ge c(\rho).$$
 Now we can use maximum principle and the matrix inequality
 $$\det(A+\lambda I) \ge \det A + n \lambda (\det A)^\frac{n-1}{n} , \quad \quad \mbox{if} \quad A \ge 0, \lambda \ge 0,$$
 and obtain in $B_{\frac 32 \rho}(z_0)$
 $$u+ \Psi \le u_0, \quad \quad u_0+ \Psi \le u, \quad \mbox{where} \quad \Psi:=C'(\rho) \, \delta_0 \ep \,  \left (|x-z_0|^2-(\frac 32 \rho)^2 \right), $$
for some large $C'(\rho)$ and this proves our claim \eqref{uu0}.

From \eqref{uu0} we find
$$|v-v_0| \le C(\rho) \delta_0 \quad \mbox{in} \quad B_{\frac 32 \rho}(z_0),$$ 
 hence $v-v_0 \to 0$ uniformly in $B_{\frac 32 \rho}(z_0)$ as $\delta_0 \to 0$.

 Next we show that, as $\eps_0 \to 0$, the corresponding $v_0$'s converge uniformly (along a subsequence) in $\ov{B_\rho}(z_0)$ to a solution of $Lw=0$.
 Note that
 $$0=\frac{1}{\ep}(\det D^2 u_0- \det D^2 U_0) = Tr (A \,  D^2v_0)$$
 where
 $$A: = \int_{0}^{1} \text{cof} (D^2 U_0 + t(D^2 u_0- D^2 U_0)) dt.$$
Here $\text{cof}~ M$ represents the cofactor matrix of $M$.

As $\eps_0 \to 0$ we have $\ep \to 0$ and therefore $D^2u_0 \to D^2 U_0$ uniformly in $\ov {B_\rho}(z_0)$. This gives $A \to \text{cof} \, \, D^2 U_0$ uniformly
in $\ov {B_\rho}(z_0)$, thus $v_0$ must converge to a solution of the linearized equation for $U_0$.

\

{\it Step 2:} We apply Lemma \ref{lineq} for the function $w$ in Step 1 and obtain
\begin{equation}\label{viq}
\left|v(x) -   (a_0 + a' \cdot x')x_{n}\right| \le \eta + C_{0} \left (|x^{'}|^2 + x_{n}^{2+\alpha}\right)^\frac{3+\alpha}{2+\alpha}
\end{equation}
for some $a_0 \in \R$, $a'= (a_1, \cdots, a_{n-1}) \in \R^{n-1}$, with $|a_0|$, $|a'|$ bounded by a constant depending on $n$ and $\alpha$.

We choose $\eta=C_0 \theta_0^\frac{3+\alpha}{2+\alpha}$ for some small $\theta_0$ to be specified later, and we find
\begin{equation}\label{unu}
\left|u-\left( \eps a_0 x_n + U_0(x'+\ep a'x_n,x_n) \right)\right| \le  \left(2 C_0 \theta_0^\frac{3+\alpha}{2+\alpha} + C \eps  \right)\ep \le 3 C_0 \theta_0^\frac{3+\alpha}{2+\alpha}\ep\quad \quad \mbox{in} \quad S_{\theta_0}(U_0). 
\end{equation}

\
 
{\it Step 3:}
Next we show that, $\nabla u(0)=0$ implies 
 \begin{equation}\label{a0_1}
 |a_0| \le C_1 \theta_0,
 \end{equation}
for some $C_1$ universal, depending only on $n$ and $\alpha$.
 
If, by contradiction, $a_0 > C_1 \theta_0$ then we construct the lower barrier
$$\underbar u:=  U_0(x'+\ep a'x_n,x_n) + \ep \left( - C'\theta_0^\frac{1}{2+\alpha}\sum_{i=1}^{n-1}(x_i+\ep a_i x_n)^2 +C'' \theta_0^\frac{1}{2+\alpha} x_n^{2+\alpha}  + \frac{a_0}{2} x_n\right),$$
with $C'=4C_0$ and $C''$ large such that
$$\det D^2 \underbar u \ge (1+  \theta_0^\frac{1}{2+\alpha }\ep)x_n^{\alpha} \ge (1+\delta_0 \ep) x_n^\alpha \ge \det D^2 u,$$
provided that $\delta_0$ is sufficiently small. We also have $\underbar u \le u$ on 
$\p(S_{\theta_0}(U_0) \cap \Om )$.

Indeed, on $\p S_{\theta_0}(U_0) \cap \Om$, in view of \eqref{unu}, it follows that
 \begin{align*}
  \underbar u  & \le  U_0(x'+\ep a'x_n,x_n) + \ep a_0 x_n - \eps \left( C'\theta_0^\frac{1}{2+\alpha}\sum_{i=1}^{n-1}(x_i+\ep a_i x_n)^2 +\frac{a_0}{4} x_n  \right) \\
  & \le U_0(x'+\ep a'x_n,x_n) + \ep a_0 x_n - \eps \left(4 C_0 \theta_0^\frac{1}{2+\alpha}|x'|^2 + \frac{C_1}{4}\theta_0 x_n + C \eps\right ) \\
  &\le U_0(x'+\ep a'x_n,x_n) + \ep a_0 x_n -3 C_0 \theta_0^\frac{3+\alpha}{2+\alpha}\ep \le u,
  \end{align*}
  and on $\p \Om \cap S_{\theta_0}(U_0)$ we use that $x_n \le \delta_0 \eps |x'|^2$ and $|a_0| \le C$ and obtain
  $$\underbar u \le \left(\frac 12 -  \frac {C'}{2} \theta_0^\frac{1}{2+\alpha} \ep + C \delta_0 \eps \right) |x'|^2  \le  \left(\frac 12 - \delta_0 \eps \right)|x'|^2 \le u.$$

By the maximum principle we obtain that $\underbar u \le u$ in 
$S_{\theta_0}(U_0) \cap \Om $ hence $u_n(0) \ge \underbar u _n (0)>0$ and we reach a contradiction.

If we assume $a_0 <- C_1 \theta_0$ then we consider the upper barrier
$$\bar u:=  \left(\frac 12 +  \eps C' \theta_0^\frac{1}{2+\alpha } \right)|x'+ \ep a'x_n|^2 + 
\frac{1-C''\eps  \theta_0^\frac{1}{2+\alpha }}{(1+\alpha)(2+\alpha)} \left((x_n -\delta_0 \ep)^+\right)^{2+\alpha}  + \frac{\ep a_0}{2} x_n.$$
Then we obtain
$$\det D^2 \bar u \le (1-\delta_0 \ep)  \left((x_n -\delta_0 \ep)^+\right)^\alpha \le \det D^2u,$$
and, similarly as above, we obtain $\bar u \ge u$ on $\p(S_{\theta_0}(U_0) \cap \Om )$. By maximum principle we find  $u  \le \bar u$ in 
$S_{\theta_0}(U_0) \cap \Om $ hence $u_n(0) \le \bar u _n (0)<0$ and again we reach a contradiction. Thus \eqref{a0_1} is proved.
\vglue 0.2cm

{\it Step 4:}  We use \eqref{a0_1} in \eqref{viq} and obtain
$$\left|v(x) -a' \cdot x' \, \, x_{n}\right| \le C_2 \theta_0^\frac{3+\alpha}{2+\alpha}  \quad \mbox{in} \quad S_{2 \theta_0}(U_0).$$
We choose $\theta_0$ small, depending also on $\beta \in (0, \frac{2}{2+\alpha})$, such that
$$C_2 \theta_0^\frac{3+\alpha}{2+\alpha} \le\frac 12 \theta_0^{1+\frac \beta 2},$$
and we easily obtain
$$|u -U_0(x'+\ep a' x_n,x_n)| \le \left(\frac 12 \theta_0^{1+\frac \beta 2} + C \ep \right) \ep \le \theta_0^{1+ \frac \beta 2} \ep\quad \mbox{in} \quad S_{2 \theta_0}(U_0). $$ 
Hence if $\ep\leq \ep_0$ is small,
$$|u(Ax)-U_0(x)| \le \theta_0^{1+\frac \beta 2} \ep\quad \mbox{in} \quad S_{\theta_0}(U_0)$$
where $Ax:=x- \ep (a', 0) x_n$. 
\end{proof}

We are now ready to give the proof of Proposition \ref{c2b}.
\begin{proof}[Proof of Proposition \ref{c2b}]

It is straightforward to check that, since $|x_n|\le \eps |x'|^2$ on $\p \Om$, the distance function $d_{\p \Om}$ satisfies
$$x_n \ge d_{\p \Om}(x) \ge \left(x_n - 2 \ep |x'|^2\right)^+.$$
We will prove the proposition for more general right hand sides in which we replace $d_{\p\Om}$ by the inequality above. Precisely we assume that $u$ satisfies
\begin{equation}\label{gen}
(1-\ep |x|^\gamma)   \left [ \left (x_n - 2\ep  |x'|^2 \right )^+ \right]^\alpha \le \det D^2 u \le (1+\ep |x|^\gamma) \, x_n^\alpha.
\end{equation}
Let $\ep_0,\delta_0,\theta_0$ be as in Lemma \ref{il}. 

We prove by induction that there exists a sequence of slidings $A_k$ such that the rescalings $u_k$ at height $h=\theta_0^k$ 
$$u_k:= \frac 1 h u(A_k F_h x), \quad x\in\tilde\Omega:= (A_k F_h)^{-1}\Om,\quad \quad \quad h=\theta_0^k,  $$
satisfy
\begin{equation}\label{ind}
|u_k-U_0| \le \bar \eps_k:=8 \delta_0^{-1} \eps \theta_0^{k\frac \beta 2}~\text{in}~S_1(U_0)\cap \tilde \Om, \quad \quad |A_k-A_{k-1}| \le C \bar \eps_{k-1}.
\end{equation}
Clearly, this holds for $k=0$ with $A_0=I$ from the hypotheses of the proposition. 

Next we assume it holds for $k$ and prove it for $k+1$. We check that $\tilde u:=u_k$ satisfies the hypotheses of 
Lemma \ref{il} for $$\tilde \eps:=\bar \eps_k \quad \quad \mbox{hence} \quad \delta_0 \tilde \eps=8 \eps h^\frac \beta 2.$$ 
By \eqref{ind}, we have $$|A_k-I| \le C \eps.$$ Note that 
$$\mbox{if} \quad \tilde x \in \overline{ \tilde \Om} \cap S_1(U_0), \quad \mbox{then} \quad \tilde x=F_h^{-1}A_k^{-1}x \quad \mbox{for some}\quad x \in \overline \Om,  $$
and 
$$x_n=h^\frac{1}{2+\alpha} \tilde x_n, \quad \quad |x'-h^\frac 12 \tilde x'| \le C \ep x_n.$$
If $\tilde x \in \p \tilde\Om\cap S_1 (U_0)$, then $x = A_k F_h \tilde x\in \p \Om$, hence $|x_n| \le \ep |x'|^2$, and we easily find
\begin{equation}\label{xbd}
 |x'-h^\frac 12 \tilde x'| \le C \ep^2 |x'|^2 \le \eps h |\tilde x'|,
\end{equation}
and
$$|\tilde x_n|\le 2 \eps h^\frac{1+\alpha}{2+\alpha}| \tilde x'|^2 \le \delta_0 \bar\eps_k |\tilde x'|^2.$$
Moreover, the boundary value of $\tilde u$ is
$$\tilde \ph( \tilde x')=\frac 1 h \ph(x')  $$
hence, using \eqref{xbd}
\begin{align*}
\left|\tilde \ph( \tilde x')-\frac 12 |\tilde x'|^2 \right| & \le 
 \frac 1 h \left| \ph( x')-\frac 12 |x'|^2 \right|+ 2 \ep h^\frac 12 |\tilde x'|^2\\
& \le \frac 1 h \ep |x'|^{2+\beta} + 2 \ep h^\frac 12 |\tilde x'|^2 \\
& \le 4 \ep h^\frac \beta 2  |\tilde x'|^2 \le \delta_0 \bar \ep_k |\tilde x'|^2.
\end{align*}
\noindent
If $\tilde x \in \tilde \Om \cap S_1(U_0)$ then we denote by $x=A_k F_h\tilde x$ and obtain
$$\det D^2 \tilde u (\tilde x)=h^{-\frac{\alpha}{2+\alpha}}\det D^2 u(x).$$
We use \eqref{gen} and $|x| \le 2 h^\frac{1}{2+\alpha}$ and obtain
$$(1-2 \ep h^\frac{\beta}{2})   \left [ \left (\tilde x_n - 8 \ep  h^\frac{1}{2+\alpha} \right )^+ \right]^\alpha \le \det D^2 \tilde u \le (1+ 2 \ep h^\frac{\beta}{2}) \, \tilde x_n^\alpha,$$
or
$$(1- \delta_0 \bar \ep_k)   \left [ \left (\tilde x_n - \delta_0 \bar \ep_k \right )^+ \right]^\alpha \le \det D^2 \tilde u \le (1+ \delta_0 \bar \ep_k) \, \tilde x_n^\alpha.$$
In conclusion we can apply Lemma \ref{il} for $\tilde u$ and obtain a sliding $\tilde A_0$, $|\tilde A_0 -I| \le C \bar \eps_k$ such that
$$\left|\frac{1}{\theta_0}\tilde u(\tilde A_0 F_{\theta_0}x) - U_0 \right| \le \tilde \eps \theta_0^\frac \beta 2 = \bar \ep_{k+1} \quad \mbox{in} \quad S_1(U_0).$$
This means that $u_{k+1}$ satisfies the desired conclusion by choosing $A_{k+1}$ such that
$$A_{k+1}F_{h\theta_0}=A_k F_h \tilde A_0 F_{\theta_0},$$
thus
$$A_{k+1}=A_k \, F_h \tilde A_0 F_h^{-1}.$$
Since $A_{k+1}$, $A_k$, $\tilde A_0$ are slidings, we obtain that the following quantities are comparable up to multiplicative constants
\begin{equation}\label{A}
|A_{k+1}-A_k| \sim |  F_h (\tilde A_0-I) F_h^{-1} | \sim h^{\frac 12 - \frac{1}{2+\alpha}}|\tilde A_0-I|,
\end{equation}
and the induction step is verified.

From \eqref{A} we conclude that $A_k$ converges geometrically to a limiting sliding $A_\infty$, thus
$$|A_\infty-A_k| \le C  h^{\frac 12 - \frac{1}{2+\alpha}} \bar \eps_k,$$ 
hence
$$A_\infty=A_k \, F_h \tilde G_k F_h^{-1} \quad  \mbox{for some sliding $\tilde G_k$ with} \quad |\tilde G_k-I| \le C \bar\eps_k.$$
This means that in \eqref{ind} we may replace the sliding $A_k$ by $A_\infty$ and the right hand side $\bar \eps_k$ by $C \bar \eps_k$,
and the proposition is proved.
\end{proof}

\section{Proof of Proposition \ref{C2b}}
\label{s4}
In this section, we prove Proposition \ref{C2b}.

Assume that $u$ satisfies the hypothesis of Proposition \ref{c2b} with
$$[g]_{C^\gamma} \le \eps \quad \mbox{in} \quad \Om \cap S_1(U_0),$$
and from the conclusion of the proposition there exists a sliding $A$, $|A-I|\le C \ep$ such that $u(Ax)$ is well approximated by $U_0$:
$$|u(Ax)-U_0(x)| \le C \ep U_0(x)^{1+\frac \beta 2}.$$
For each $h<c$ we let $-t$ be the minimum value of $u-h^\frac{1+\alpha}{2+\alpha}x_n$ and $x_t$ the point where is achieved. Thus
 $$S_t(x_t)=\{ u<h^\frac{1+\alpha}{2+\alpha}x_n \}$$
is a section centered at $x_t$ which is tangent to $\p \Om$ at $0$. It is easy to check that $t \sim h$ and the sections $S_h(u)$ and $S_t(x_t)$ are comparable in size. Below we estimate the modulus of continuity of $D^2u$ in the interior of $S_{t}(x_t)$.

\begin{lemma}\label{L_1}
 Assume $\beta \le \alpha$. In the section $S_{t/4}(x_t)$ we have
 $$ \|D^2 u\|_{C^\beta} \le C, \quad \| D^2u-D^2u(0)-x_n^\alpha e_n \otimes e_n\|_{L^\infty}\le C \eps h^\frac \beta 2.$$
\end{lemma}

\begin{proof}
 Using that $u(Ax)$ is well approximated by $U_0(x)$ we obtain 
 $$S_{t/2}(x_t) \subset \mathcal C:= \left \{|x'| \le  |x_n|   \right \}.$$
 Since $\p \Om \subset \{|x_n| \le \eps |x'|^2\}$ it follows that in the cone $\mathcal C$, $\frac{d_{\p \Om}}{x_n}$ is a positive function with gradient bounded by $1+ C \eps$, and in fact
$$\left \| \frac{d_{\p \Om}}{x_n}-1 \right \|_{C^{0,1}(\mathcal C)} \le C \eps.$$
This means that in the cone $\mathcal C$ (and therefore in $S_{t/2}(x_t)$) we can rewrite the equation for $u$ as
 $$\det D^2 u = \bar g \, \, x_n^\alpha, \quad \quad \|\bar g -1 \|_{C^\gamma}\le C \eps \quad \mbox{in $\mathcal C$}.$$
 We let $u_h$ denote the rescaled function 
$$u_h(x)=\frac 1 h u(A F_h x),$$
 and let $y_t$ such that $ x_t=A F_h y_t$. Notice that
 $$ A F_h  S_{t/(2h)}(y_t) = S_{t/2}(x_t), \quad \quad S_{t/h}(y_t)=\{u_h <y_n\}$$
where $S_t(y)$ denote the sections for $u_h$.

We have $$\det D^2 u_h= \bar g_h \, \, x_n^\alpha \quad \mbox{in}\quad S_{t/(2h)}(y_t),$$
with $$\bar g_h(x)=\bar g (AF_h x) \quad \quad \Rightarrow \quad \quad \|\bar g_h-1\|_{C^\gamma} \le  C \eps h^\frac {\gamma}{2+\alpha}=C\eps h^\frac \beta 2.$$
From Proposition \ref{c2b} we have $$|u_h -U_0| \le C \eps h^\frac \beta 2 \quad \mbox{in} \quad  S_{t/(2h)}(y_t).$$
Notice that $\gamma > \beta$ hence, by the interior $C^{2,\beta}$ estimates for the Monge-Amp\`ere equation, we obtain
$$\|D^2 (u_h- U_0) \|_{C^\beta} \le C \eps h^\frac \beta 2 \quad \quad \mbox{in} \quad S_{t/(4h)}(y_t).$$
Rescaling back, we write these inequalities in terms of $D^2u$ and obtain that in $S_{t/4}(x_t)$
$$\|D^2 (u-U_0(A^{-1}x))\|_{L^\infty} \le C \eps h^\frac \beta 2, \quad \quad  \|D^2 (u-U_0(A^{-1}x))\|_{C^\beta} \le C \eps.$$
We can write $$D^2(U_0(A^{-1}x))=Q_0 + x_n^\alpha e_n \otimes e_n \quad \quad \mbox{ and} \quad  Q_0=D^2u(0).$$
Now we use that $\beta \le \alpha$ hence $x_n^\alpha \in C^\beta$ and the proof is completed.
\end{proof}

Next we check that Proposition \ref{c2b} can be applied at the origin also when $u(0)$ and $\nabla u (0)$ do not necessarily vanish provided that  $\p \Om$ is $C^{2,\beta}$. Thus the proposition can be used at all other boundary points $z \in \p \Om \cap S_{1/2}(U_0)$.
\begin{lemma}\label{L_2}
Assume that $u$ solves $$\det D^2 u=g \, d^\alpha_{\p \Om} \quad \mbox{in} \quad \Om \cap S_1(U_0),$$
and
$$|u-U_0| \le \eps, \quad [g]_{C^\gamma}\le \eps, \quad u(x)=\varphi(x')~\text{on}~\p\Om, \quad [D^2 \ph]_{C^\beta} \le \eps, $$
$$\p \Om= \{x_n=q(x')\} \quad {with} \quad  \|q\|_{C^{2, \beta}}\le \eps.$$
Then there exist 2 linear transformations, $A$ sliding, $D$ symmetric, which are $C \eps$ perturbations of the identity $I$ such that
$$|(u-l_0)(D A x)-U_0(x)| \le C \eps U_0^{1+\frac \beta 2},$$
where $l_0(x)=u(0)+\nabla u(0) \cdot x$ represents the linear part of $u$ at $0$.
\end{lemma}

\begin{proof}
We sketch the proof below.

Since $|u-U_0| \le \eps$ we easily find $$|g(0)-1|\le C \eps, \quad |D_{x'}^2 \ph(0)-I' |\le C \eps, \quad |\nabla_{x'}\ph(0)| \le C \eps.$$
Using standard barriers at the origin we find $|u_n(0)| \le C \eps$ hence the linear part $l_0$ of $u$ at the origin satisfies $$l_0(x)=u(0)+\nabla_{x'}\ph(0) x'+u_n(0)x_n = O(\eps).$$ 
The function $w=u-l_0$ satisfies the same equation as $u$ and $w(0)=0$, $\nabla w(0)=0$ and its boundary values on $\p \Om$ equal
$$\ph_w(x')=\ph(x')-l_0(x',q(x')) \quad \mbox{hence} \quad \|D_{x'}^2 \ph_w-I'\|_{C^{2,\beta}} \le C \eps.$$
Now we may choose a symmetric matrix $D$ which has $e_n$ as an eigenvector such that 
$$\tilde w(x)=w(D x), \quad x\in \tilde \Om:= D^{-1}\Om$$
satisfies the hypotheses of Proposition \ref{c2b} for its boundary data $\ph_{\tilde w}$ and right hand side $g_{\tilde w}d_{\p\tilde\Om}^{\alpha}$ with
$$D^2_{x'}\ph_{\tilde w}(0)=I', \quad \quad g_{\tilde w}(0)=1.$$

In conclusion there exists a sliding $A$ which is a $C\ep$ perturbation of the identity $I$ such that $|\tilde w(A x)-U_0(x)| \le C \eps U_0(x)^{1+\frac \beta 2}$.
\end{proof}

We write $U_0(A^{-1}D^{-1}x)$ as a quadratic polynomial plus a multiple of $x^{2+\alpha}$. 

Since $U_0(A^{-1}D^{-1}x) \le C |x|^2$, the conclusion in the lemma above gives 
\begin{equation}\label{u-l}
\left|u-\left (l_0 + \frac 12 x^TD^2u(0)x + \frac{a(0)}{(1+\alpha)(2+\alpha)}x_n^{2+\alpha}\right)\right | \le C \eps |x|^{2+\beta}.
\end{equation}
The coefficient $a(0)$ and the tangential determinant of $D^2_{x'}u$ satisfy
\begin{equation}\label{a0}
a(0) \det D_{x'}^2u(0)=g(0).
\end{equation}
Clearly the conclusion holds for all points $z \in \p \Om \cap S_{1/2}(U_0)$ i.e.
$$\left|u-\left (l_z + \frac 12 (x-z)^TD^2u(z)(x-z) + \frac{a(z)}{(1+\alpha)(2+\alpha)}((x-z)\cdot \nu_z)^{2+\alpha}\right)\right | \le C \eps |x-z|^{2+\beta}.$$
Next we show that $D^2u$ and $a$ are $C^\beta$ on $\p \Om \cap S_{1/2}(U_0)$.
\begin{lemma}\label{L_3} On $\p \Om \cap S_{1/2}(U_0)$, we have
$$[D^2u]_{C^\beta} \le C \eps, \quad \quad [a]_{C^\beta} \le C \eps.$$
\end{lemma}

\begin{proof}
Since
$$d_{\p \Om} \le x_n \le d_{\p \Om}+ C \eps |x|^2,$$
it follows that
$$|x_n^{2+\alpha}-d^{2+\alpha}_{\p \Om}| \le C \ep |x|^{2}d^{1+\alpha} + C (\ep |x|^2)^{2+\alpha} \le C \ep |x|^{3+\alpha}.$$
Also, from the previous lemma we have $|a(0)-1| \le C \eps$. Using these in \eqref{u-l} together with $\beta \le \alpha$ and we find
$$\left|u-\left (l_0 + \frac 12 x^TD^2u(0)x + \frac{1}{(1+\alpha)(2+\alpha)}d_{\p \Om}^{2+\alpha}\right)\right | \le C \eps |x|^{2+\beta}.$$
Subtracting the corresponding inequality at $z$ we find
$$\left|l_0 -l_z + \frac 12 x^TD^2u(0)x - \frac 12 (x-z)^TD^2u(z)(x-z)\right | \le C \ep |z|^{2+\beta} \quad \quad \forall x \in B_{|z|}\cap \Om$$
which gives 
$$|D^2u(0)-D^2u(z)| \le C \ep |z|^\beta.$$
Now the corresponding inequality for $a(z)$ follows easily from the equality \eqref{a0} and the fact that $[g]_{C^\beta} \le C \eps$. 
\end{proof}

\
We now  give the proof of Proposition \ref{C2b}.
\begin{proof}
[Proof of Proposition \ref{C2b}]

Let $x_0$ and $x_1$ be two points in $\Om\cap S_{1/4}(U_0)$ and let $S_{t_0}(x_0)$, $S_{t_1}(x_1)$ be the the maximal sections included in $\Om$ which become tangent to $\p \Om$ at, say $0$ for simplicity, and respectively $z$. 
Let $d_0:=|x_0|$, $d_1:=|x_1-z|$ and assume $d_0 \ge d_1$. Next we show that
$$|D^2u(x_1)-D^2(x_0)| \le C |x_1-x_0|^\beta.$$
 We distinguish 3 cases depending on the positions of $x_0$, $x_1$ and $z$.

\

{\it Case 1:} $x_1 \in S_{t_0/4}(x_0)$. Then we're done by Lemma \ref{L_1}.

\

{\it Case 2:} $x_1 \notin S_{t_0/4}(x_0)$ and $|z|^2 \le c \, d_0^{2+\alpha}$.

Then 
$|x_0-x_1| \ge c d_0$. 
This inequality follows by inspecting the corresponding sections for the rescaling $u_h$ with $h=d_0^{2+\alpha}$ and noticing that their centers $y_0$ and $y_1$ must 
satisfy $(y_0-y_1)\cdot e_n \ge c|y_0|.$ We rescale back and obtain 
$|x_0-x_1| \ge c d_0.$

Moreover by Lemma \ref{L_1} and $\alpha\geq \beta$ we have
$$|D^2u(x_0)-D^2u(0)| \le C (x_0 \cdot e_n)^\alpha + C \ep d_0^{(2+\alpha)\frac \beta 2} \le C d_0^\beta,$$
and similarly
$$|D^2u(x_1)-D^2u(z)| \le C d_1^\beta.$$
Now we obtain the desired inequality since by Lemma \ref{L_3}, $|D^2 u(z)-D^2u(0)| \le C \ep z ^\beta \le C d_0^\beta$.

\

{\it Case 3:}  $|z|^2 \ge c\, d_0^{2+\alpha}.$

Then $|x_0-x_1| \ge c |z|$. This follows by inspecting the images of the sections under the rescaling $u_h$ with $h=|z|^2$ and noticing 
that their centers $y_0$ and $y_1$ must satisfy $|y_0-y_1| \ge c$.

We use Lemma \ref{L_1} at $0$ and $z$ and obtain
$$D^2 u(x_0)=D^2u(0) + a(0) \, (x_0 \cdot e_n)^\alpha \, e_n \otimes e_n + O(\ep |z|^\beta),$$
$$D^2 u(x_1)=D^2u(z) + a(z)\, ((x_1-z) \cdot \nu_z)^\alpha \, \nu_z \otimes \nu_z + O(\ep |z|^\beta).  $$
 Now the the desired result is a consequence of Lemma \ref{L_3}, $|x_0-x_1| \ge c |z|$, and
 $$|e_n-\nu_z| \le C \ep |z|,$$  $$ |x_0 \cdot e_n-(x_1-z) \cdot \nu_z| \le C |x_0-x_1| + C|z| \le C |x_0-x_1|.$$

\end{proof}

\section{Global $C^{\infty}$ regularity of Monge-Amp\`ere eigenfunctions}
\label{s5}
In this section, we prove global $C^{\infty}$ regularity of Monge-Amp\`ere eigenfunctions stated in Theorem \ref{T035}. For convenience, we restate it here. 
\begin{theorem}\label{T04} Let $\Om$ be a bounded and uniformly convex domain in $\R^n$.
Assume $\p \Om \in C^{\infty}$ and $u$ satisfies  
\begin{equation}\label{EP1}
\left\{
 \begin{alignedat}{2}
   (\det D^2 u)^{\frac{1}{n}} ~& = \lambda|u|~&&\text{in} ~  \Omega, \\\
u &= 0~&&\text{on}~ \p\Omega.
 \end{alignedat} 
  \right.
  \end{equation}
 Then $u \in C^{\infty}(\ov \Om)$.
\end{theorem}
To prove the above theorem, we perform first a Hodograph transform and reduce (\ref{EP1}) to a similar equation in the upper half-space. Then we perform the partial Legendre transform in the nondegenerate 
$x'$ coordinates. The structure of the equation satisfied by the transformed function allows us to use the $C^{2,\beta}$ estimates for the 
 Monge-Amp\`ere eigenfunctions obtained in Theorem \ref{T03} together with Schauder estimates for linear equations with H\"older coefficients modeled 
by a degenerate Grushin-type operator as in (\ref{w-eq}).
 These give the desired global $C^{\infty}$ regularity.

\subsection{Equivalent equation in the upper half-space}
We first write an equation in the upper half-space that is locally equivalent to (\ref{EP1}). After a dilation we may assume that $\lambda=1$. 

The Monge-Amp\`ere eigenfunctions are $C^{\infty}$ in the interior of $\Om$ so it remains to prove their $C^{\infty}$ smoothness near the boundary $\p\Om$.
Assume that $0\in\partial\Omega$ and 
$e_n$ is the inner normal of $\partial\Omega$ at $0$. We make the rotation of coordinates
$$y_{n}=- x_{n+ 1},~ y_{n+ 1}= x_n,~ y_k= x_k~(1\leq k\leq n-1).$$
In the new coordinates, the graph of $u$ near the origin can be represented as 
$y_{n+1} = \tilde u(y)$ in the upper half-space $\R^{n}_{+}=\{y\in \R^n: y_{n}>0\}$.
The Gauss curvature is
$$K= \frac{\det D^2 u}{(1+ |Du|^2)^{\frac{n+2}{2}}}= \frac{ |u|^n}{(1+ |Du|^2)^{\frac{n+2}{2}}}.$$
The tangent plane at $0$ is given by
$$x_{n+1}- u_{n}x_n - u_k x_k=0.$$
After the above rotation of coordinates, it is given by
$$-y_n -u_n y_{n+1}-u_k y_k=0,~\text{or}~~y_{n+1} + \frac{y_n}{u_n} + \frac{u_k}{u_n} y_k=0.$$
Hence
$$\tilde u_{y_n}= -\frac{1}{u_{x_n}}>0~(\text{near}~0), ~ \tilde u_{y_k}= -\frac{u_k}{u_n}.$$
Now, from (\ref{EP1}) we obtain the following equation in a neighborhood of the origin in the upper half-space $\{y_n>0\}$:
$$\det D^2 \tilde u = K (1+ |D\tilde u|^2)^{\frac{n+2}{2}}= |y_n|^{n} \left(\frac{1+ |D\tilde u|^2}{1+ |Du|^2}\right)^{\frac{n+2}{2}}=
 |y_n|^n (\frac{1}{u_n^2})^{\frac{n+2}{2}}=  y_n^n \, \, \tilde u_n^{n+2}.$$
Near the origin, the boundary $\p\Om$ is given by $x_n = \phi(x')$ in the original coordinates. Thus, the boundary condition for $\tilde u$ is $\tilde u=\phi$ on $\{y_n=0\}.$ Thus, locally, we have for some small $r_0>0$
(now relabeling $y$ by $x$)
\begin{equation}\label{ET}
\left\{\begin{array}{rl}
\det D^2 \tilde u  &=  x_n^n \, \, \tilde u_n^{n+2} \quad \mbox{in}~ B_{r_0}^{+},\\
\tilde u&= \phi \quad \quad \quad \quad \mbox{on}~\{x_n=0\}\cap B_{r_0}.
\end{array}\right.
\end{equation}

We remark that since $u \in C^{2,\beta}(\ov \Om)$ by Theorem \ref{T03}, we have $\tilde u \in C^{2,\beta}(\ov B_{r_0}^+) $ for some small $\beta>0$, 
and $\tilde u_n >c$. It remains to show that solutions of \eqref{ET} with $\phi \in C^\infty$ are smooth up to the boundary in a neighborhood of the origin.
 For simplicity of notation, we relabel $\tilde u$ from (\ref{ET}) by $u$.

\subsection{Partial Legendre transform} \label{Leg_sec}Next we perform the following partial
Legendre transformation to the solutions of \eqref{ET}:
\begin{equation}\label{leg}
y_i= u_i(x) \quad (i\le n-1),\quad  y_n = x_n,\qquad u^*(y)=x' \cdot \nabla_{x'}u -u(x).
\end{equation}
The function $u^*$ is obtained by taking the Legendre transform of
$u$ on each slice $x_n=\text{constant}.$ 
Note that $$(u^*)^* =u.$$
We claim that if $u$ satisfies (\ref{ET}) then $u^*$ (which is convex
in $y'$ and concave in $y_n$) satisfies
\begin{equation}\label{eqn-u*}
\left\{\begin{array}{rl}
y_n^{\alpha} (-u^*_n)^{n+ 2} \det D^2_{y'}u^* + u^*_{nn}  &= 0~\mbox{in}~ B_\delta^{+}\\
u^*&= \phi^*~~~~\mbox{on}~\{y_n=0\}\cap B_\delta,
\end{array}\right.
\end{equation}
where $\alpha=n$. Moreover $u^* \in C^{2,\beta}(\ov B_\delta^+)$, $-u_n^*>c$ and $\phi^* \in C^\infty$.

There are several ways to see this. A direct way is to compute the transformation matrices $Y = (Y_{ij}) = \left(\frac{\p y_i}{\p x_j}\right)$ and $Y^{-1}=X= (X_{ij})
= \left(\frac{\p x_i}{\p y_j}\right)$. We obtain
$$Y_{ij}= 
\left\{\begin{array}{rl}
u_{ij} & \text{if}~~ i\leq n-1\\
\delta_{jn}&\text{if}~~ i=n
\end{array}\right.,
~X= \left[
\begin{array}{c|c}
(D^2_{x'} u)^{-1} & U^{ni}/\det D^2_{x'} u\\ \hline
0 & 1
\end{array}\right]$$
where $(U^{ij})$ denotes the cofactor matrix of $D^2 u$.

Using these transformations, we find that $$u^*_{j}= x_j~ \text{if}~ j<n,~u^*_n= -u_n$$ and therefore
$$D^2_{y'} u^*= (D^2_{x'}u)^{-1},~
 u^*_{nn} = -\sum_{k\leq n}\frac{\p u_n}{\p x_k}\frac{\p x_k}{\p y_n}= -\sum_{k<n} u_{nk}\frac{U^{nk}}{\det D^2_{x'} u}-u_{nn}.
$$
Since $\det D^2_{y'} u^{\ast} = (\det D^2_{x'} u)^{-1} = 1/U^{nn}$,
it follows that
\begin{equation*}
 \frac{-u^*_{nn}}{\det D^2_{y'} u^{\ast}} = \left( \sum_{k<n} u_{nk}\frac{U^{nk}}{\det D^2_{x'} u}+ u_{nn}\right)\det D^2_{x'} u= \sum_{k\leq n}u_{nk}U^{nk} = \det D^2 u.
\end{equation*}
Hence $u^*$ satisfies \eqref{eqn-u*}. It remains to prove that $u^*$ is smooth up to the boundary in a neighborhood of the origin.

\subsection{Schauder estimates for linear equations}
In order to obtain the smoothness of $u^*$ from \eqref{eqn-u*} we establish Schauder estimates for its linearized equation. We consider linear equations of the form
\begin{equation}x_n^{\alpha} \sum_{i, j\leq n-1}a^{ij} v_{ij} + v_{nn}= x_n^{\alpha} f(x)
 \label{veq}
\end{equation}
with $a^{ij}$ uniformly elliptic $$\lambda I' \le (a^{ij}(x))_{i,j} \le \Lambda I'.$$
\begin{definition}\label{dal}
We define the distance $d_\alpha$ between 2 points $y$ and $z$ in the upper half-space by
$$d_\alpha(y,z):=|y'-z'| +\left|y_n^\frac{2+\alpha}{2}-  z_n^\frac{2+\alpha}{2}\right|.$$
\end{definition}
Notice that $d_\alpha$ is equivalent to the distance induced by the metric 
$$ds^2 = dx'^2 + x_n^\alpha dx_n^2.$$
The relation between $d_\alpha$ and the Euclidean distance in $\overline{B}_1^+$ is as follows:
\begin{equation}\label{51}
c|y-z|^\frac{2+\alpha}{2} \le d_\alpha(y,z) \le C |y-z| ,
\end{equation}
$$ d_\alpha(y,z) \sim |y-z|  \quad \mbox{if} \quad y,z \in \overline{B}_1^+\cap \{x_n \ge 1/4  \}.$$
The distance $d_\alpha$ occurs naturally in equation \eqref{veq} since they both scale homogeneously after the transformations $x \mapsto F_h x$ (see Section \ref{res_sec}). 

If function $w$ is $C^\gamma$ with respect to $d_\alpha$ (with $\gamma \in (0, \frac{2}{2+\alpha})$) we write
$$w \in C_\alpha^\gamma(\ov B_1^+),$$ 
and define
$$[w]_{C_\alpha^\gamma(\ov B_1^+)}= \sup_{y, z\in \ov B_1^+, ~y \ne z}\frac{|w(y)-w(z)|}{( d_\alpha(y,z))^\gamma}, 
\quad \|w\|_{C_\alpha^\gamma(\ov B_1^+)}=\|w\|_{L^\infty(\ov B_1^+)}+ [w]_{C_\alpha^\gamma(\ov B_1^+)}.$$
In view of \eqref{51} we obtain the following relations between the $C^\gamma_\alpha$ spaces and the standard $C^\gamma$ Holder spaces : 
\begin{align*}
 w \in C^\gamma_\alpha(\ov B_1^+)  \quad & \Rightarrow \quad w \in C^\gamma (\ov B_1^+) \\
 w \in C^\beta (\ov B_1^+)  \quad &  \Rightarrow \quad w \in C^\gamma_\alpha(\ov B_1^+) \quad \mbox{with} \quad \gamma=\beta \, \,  \frac{2}{2+\alpha}.
\end{align*}

\begin{proposition} [Schauder estimates]
 \label{LMA-model1}
 Assume that $v$ solves \eqref{veq} in $\ov B_\delta^+$ and
$$v= \varphi(x')~\text{on}~ \{x_n=0\}\cap \ov B_\delta^+.$$
If $a^{ij}$, $f \in C_\alpha^\gamma(\ov B_\delta^+)$ with $\frac{\gamma}{2} \le \frac{\min\{1,\alpha\}}{2+\alpha}$, and $\varphi \in C^{2,\gamma}$, then
$$D v, D^2v \in C^\gamma_\alpha(\ov B_{\delta/2}^+).$$
\end{proposition}

The proof of this proposition is standard and we postpone it till the last section.\\
By repeatedly differentiating (\ref{veq}) in the $x'$ direction we easily obtain Schauder estimates for higher derivatives. Below $m=(m_1,..,m_{n-1})$ denotes a multi-index with $m_i$ nonnegative integers.
\begin{corollary}\label{c0}
If in the proposition above $\varphi \in C^{k+2,\gamma}$ for some integer $k \ge 0$ and
$$D^m_{x'}a^{ij}, D^m_{x'}f \in C^\gamma_\alpha (\ov B_\delta^+) \quad \forall \, \, m \quad \mbox{with} \quad |m| \le k,$$
then
$$D D^m_{x'} v, D^2 D^m_{x'}v \in C^\gamma_\alpha(\ov B_{\delta/2}^+) \quad   \quad \forall \, \,m \quad \mbox{with} \quad   |m| \le k.$$
\end{corollary}

\subsection{Proof of Theorem \ref{T04}} 
As mentioned at the end of Section \ref{Leg_sec}, it suffices to prove $C^{\infty}$ regularity in a neighborhood of the origin for the function $u^*$ satisfying \eqref{eqn-u*}:
\begin{equation*}
\left\{\begin{array}{rl}
y_n^{\alpha} (-u^*_n)^{n+ 2} \det D^2_{y'}u^* + u^*_{nn}  &= 0 \quad \quad \quad ~\mbox{in}~ B_\delta^{+},\\
u^*&= \phi^*~~~~\quad \quad \mbox{on}~\{y_n=0\}\cap B_\delta,
\end{array}\right.
\end{equation*}
with $\alpha=n$, $u^* \in C^{2,\beta}(\ov B_\delta^+)$, $-u_n^*>c$ and $\phi^* \in C^\infty$.

Fix $k<n$. Then $v= u^*_k$ solves the linearized equation
\begin{equation}
 \label{LMA_uk}
y_n^{\alpha} \sum_{i, j\leq n-1}a^{ij}  v_{ij} + v_{nn}= y_n^{\alpha} f(y)\quad\mbox{in}\quad B_\delta^+
\end{equation}
where
$$a^{ij} =(-u^*_n)^{n+ 2} U_{y'}^{\ast ij}, \quad \quad \quad f(y) = (n+2)(-u_n^{\ast})^{n+1} u^*_{nk}\det D^2_{y'}u^*.$$
and $U^*_{y'}$ denotes the cofactor matrix of $D^2_{y'}u^*$.

Since $u^* \in C^{2,\beta}(\ov B_\delta^+)$ we obtain $Du^*$, $D^2u^* \in C^\gamma_\alpha(\ov B_\delta^+)$ for some small $\gamma>0$, hence $a^{ij}$, $f \in C^\gamma_\alpha(\ov B_\delta^+)$.

By Proposition \ref{LMA-model1},
$D^2v \in C^\gamma_\alpha$ up to the boundary in $\ov B_{\delta/2}^+ $ which in turn implies $D_{y'}a^{ij}$, $D_{y'}f \in C^\gamma_\alpha (\ov B_{\delta/2}^+)$. Now 
we may apply Corollary \ref{c0} and iterate this argument to obtain that  $D^m_{y'}D^l_{y_n}u^*$ with $l\in \{0,1,2\}$ are continuous up to the boundary 
in $\ov B_{\delta/2}^+ $ for all multi-indices $m \ge 0$. In order to obtain the continuity of these derivatives for all values of $l$ we differentiate the 
equation for $u^*$ and use that $\alpha=n$ is a nonnegative integer. Then each derivative $D^m_{y'}D^l_{y_n}u^*$ with $l \ge 3$ can be expressed as a polynomial 
involving powers of $y_n$ and derivatives $D_{y'}^q D_{y_n}^su^*$ with $s< l$, thus $u^* \in C^\infty (\ov B_{\delta/2}^+)$ as desired. 
\qed

\section{Rescaling and the linearized equation}
\label{res_sec}
In this section, we prove some technical results (Lemmas \ref{good_rescale} and \ref{lineq} and Proposition \ref{LMA-model1}) which were used in the proofs of our main results in 
the previous sections.
\subsection{Rescaling} We give below the proof of Lemma \ref{good_rescale}.

Suppose that $u$ satisfies the hypotheses of Theorem \ref{T2.2}. Assume furthermore that the boundary $\p\Om$ is locally given by the graph of a $C^{1,1}$ convex function $q(x')$, i.e.
$$\p\Om\cap B_{\rho}=\{x_n = q(x')\}, \quad \quad q(0)=0, \quad \nabla_{x'}q(0)=0, \quad \|q\|_{C^{1,1}}\leq \frac{1}{\rho'}.$$

As mentioned in the Introduction, Lemma \ref{good_rescale} is a direct consequence of the pointwise $C^2$ estimate form \cite{S2}. 
We only need to check that, after an initial rescaling, the equation remains of the same form. 
First we need the following lemma which compares distance functions during the rescaling process.

\begin{lemma}
\label{ddh} 
Let $A$ be a sliding along $x_n=0$ with $|A| \le C.$
Let $\Omega_h = (AF_h)^{-1}\Omega.$
If $h\leq c$ then
\begin{myindentpar}{1cm}
(i) $\p \Omega_h \cap S_1(U_0)$ is a graph 
in the $e_n$ direction whose $C^{1,1}$ norm 
is bounded by $C h^{\frac{1+\alpha}{2+\alpha}}$.\\
(ii) the ratio of the distance functions to $\p \Om$ and $\p \Om_h$ is Lipschitz and
\begin{equation*}
 \left [ \frac{d_{\p\Omega}(A F_h x)}{h^{\frac{1}{2+\alpha}} d_{\p\Omega_h}(x)}\right]_{C^{0,1}}\leq  Ch^{\frac{1}{2+\alpha}} \quad \quad \mbox{in} \quad \Om_h \cap S_1(U_0).
\end{equation*}
\end{myindentpar}
\end{lemma}

We remark that the ratio above tends to 1 as $x \to 0$.

\begin{proof} 
(i) The $C^{1,1}$ norms corresponding to $\Om$ and $A^{-1} \Om$ are comparable, since $A$ is bounded. Now the conclusion follows after we apply the linear transformation $F_h^{-1}$ to $A^{-1} \Om$. One way to see this is to think of $F_h^{-1}$ as a composition of a dilation of factor $h^{-\frac 12}$ followed by a 
$h^{\frac 12 -\frac {1}{2+\alpha}}$ contraction along the $x_n$ direction. 

(ii) The Lipschitz estimate follows from Lemma \ref{C11} below which estimates the Lipschitz norm of the ratio of two positive $C^{1,1}$ functions vanishing on the boundary. We apply  Lemma \ref{C11} to $\Om_h$, $v(x)=d_{\p\Om_h}(x),$
$u(x)= h^{-\frac{1}{2+\alpha}} d_{\p \Omega}(X)$ with $X=AF_hx$, and $\ep= Ch^{\frac{1}{2+\alpha}}$ to obtain the desired conclusion. 

Indeed, by (i), the curvatures of $\p\Om_h\cap S_1(U_0)$ are bounded by $Ch^{\frac{1+\alpha}{2+\alpha}}\leq \ep$, thus $|D^2v| \le C \eps$. Furthermore, since $|D^2_X d_{\p \Om}| \le C$, we obtain 
$$|D^2 u(x)|=h^{-\frac{1}{2+\alpha}} |F_h^T A^T D^2_X d_{\p \Om} A F_h| \le C h^\frac 1{2+\alpha}.$$
The proof of (ii) is complete.
\end{proof}

Below we present a general Lipschitz bound for the quotient of two positive $C^{1,1}$ functions vanishing continuously on the boundary os some $C^{1,1}$ domain.
\begin{lemma}\label{C11}
Assume that $\Om$ is a $C^{1,1}$ domain with norm bounded by $\eps$, and suppose $0\in\p\Om$. 
Let $u, v$ be positive  $C^{1,1}$ functions in $\Om\cap B_1$ that vanish continuously on $\p\Om\cap B_1$ and assume $v_{\nu}\geq c$ on $\p\Om\cap B_1$, and in $\Om\cap B_1$ we have
$$|\nabla u|\leq C, \quad |\nabla v|\leq C, \quad [\nabla u]_{C^{0, 1}}\le \eps, \quad [\nabla v]_{C^{0, 1}}\leq \ep.$$
Then
$$\left[\frac{u}{v}\right]_{C^{0, 1}}\leq C\ep~\text{in}~ \Om\cap B_{1/2}.$$
\end{lemma}

\begin{proof} For $x\in \Om\cap B_{1/2}$, let $\xi_x=\nabla d_{\p \Om}(x) $ be the unit vector at $x$ which gives the perpendicular direction to $\p\Om$. Then
\begin{equation*}
\frac{u(x)}{d_{\p\Om}(x)} = \frac{1}{d_{\p\Om}(x)}\int_1^0 \left[\frac{d}{dt} u( x - td_{\p\Om}(x) \xi_x)\right]dt= \int_0^1 \nabla u ( x - td_{\p\Om}(x) \xi_x) \cdot \xi_x \, dt.
\label{ud}
\end{equation*}
The bound on the curvatures of $\p\Om$  implies
$\displaystyle \left[\xi\right]_{C^{0,1}}\leq \ep.$
This combined with the hypotheses on $\nabla u$ easily gives
\begin{equation*}
\left[\frac{u}{d_{\p\Om}}\right]_{C^{0, 1}} \leq C\ep \quad \mbox{in} \quad \Om\cap B_{1/2}.
\label{ude}
\end{equation*}
The same holds for $v$, and since $v_\nu \ge c$ on $\p \Om$ we find
\begin{equation*}
\left[\frac{d_{\p\Om}}{v}\right]_{C^{0, 1}}\leq C\ep.
\label{dve}
\end{equation*}
and the conclusion of the lemma follows.
\end{proof}

\begin{proof}[Proof of Lemma \ref{good_rescale}] 
First, we consider the case when $D^2_{x'}\ph(0)=I'$, and $g(0)=1$. Then we can apply the pointwise $C^2$ estimate, Theorem 2.4 in \cite{S2}. The theorem states that there exists 
a sliding $$A \, x= x + \tau x_n, ~\text{with} ~|\tau| \le C(\rho,\rho')$$ such that for any $\eta>0$, we have
\begin{equation*}
(1-\eta) A \, \, S_h(U_0) \subset S_h(u) \subset (1+\eta) A \, \, S_h(U_0).
\end{equation*}
for all $h \leq \bar c(\eta)$.
Thus $u_h(x)=h^{-1}u(AF_hx)$ satisfies 
$$|u_h-U_0|\leq  C\eta~
\text{in}~ S_1(U_0)\cap\Om_h$$ for some $C$ universal.
By choosing $\eta=\ep_0/C$, we obtain for all $h\leq \bar{c}(\ep_0)$
$$|u_h-U_0|\leq  \ep_0~
\text{in}~ S_1(U_0)\cap\Om_h.$$

Next, we check that $q_h$ and the boundary data $\varphi_h$ satisfy on $S_1(U_0)\cap \p\Om_h$
\begin{equation}[D^2\varphi_h]_{C^{\beta}}\leq Ch^{\frac{\beta}{2}}[D^2\varphi]_{C^{\beta}},\quad \quad \|q_h\|_{C^{2, \beta}}\leq Ch^{\frac{1+\alpha}{2+\alpha}}\|q\|_{C^{2, \beta}},
 \label{bound_phi_h}
\end{equation}
and hence they can be made smaller than $\ep_0$ if we choose $h$ small.

Let $x\in \p\Om_h \cap S_1(U_0)$, and let $y= AF_h \, x\in\p\Om$ and $z=F_h x$. Thus
$$ h^{\frac{1}{2+\alpha}}x_n=z_n, \quad \quad h^{\frac 12} x'=z', \quad \quad y=Az=z+\tau z_n.$$
Let $\bar \varphi(z')=\varphi(y')$ and $\bar q(z')=q(y')$ denote the corresponding functions written in the $z$ variables. Since $A, A^{-1}$ are bounded linear transformations we find that the $C^{2,\beta}$ norms of $\bar \varphi$ and $\varphi$, respectively $\bar q$ and $q$, are comparable.

On the other hand we have
\begin{equation*}
\varphi_h(x')= \frac{1}{h}\bar \varphi(z'),\quad q_h(x')= h^{-\frac{1}{2+\alpha}} \bar q(z').
\end{equation*}
and \eqref{bound_phi_h} easily follows since $z'=h^{\frac 12} x'$.

Finally, we verify the behavior of $g$ after the rescaling $AF_h.$ 

Let $x \in \Omega_h\cap S_1(U_0)$,
and let $y:= AF_h x\in\Omega.$
Then
$$\det D^2 u_h(x) = h^{\frac{2}{2+\alpha}-1} \det D^2 u(y) = g(y) \, d_{\p\Omega}^{\alpha}(y) \, h^{-\frac{\alpha}{2+\alpha}}:= g_h(x) d^{\alpha}_{\p \Om_h}(x),$$
where 
$$g_h(x)= g(y) \left(\frac{d_{\p\Omega}(A F_h x)}{h^{\frac{1}{2+\alpha}} d_{\p\Omega_h}(x)}\right)^{\alpha}.$$
Now we may apply Lemma \ref{ddh} together with $|D_xy|=|AF_h| \le Ch^\frac{1}{2+\alpha}$ and obtain 
$$ g_h(0)=g(0)=1, \quad \quad [g_h]_{C^{\gamma}}\leq Ch^{\frac{\beta}{2}}.$$

In the general case, we perform an initial linear transformation given by a symmetric matrix $D$ which has $e_n$ as an eigenvector such that 
$$\tilde u(x)=u(D x), \quad x\in \tilde \Om:= D^{-1}\Om,$$
satisfies $D^2_{x'}\tilde \ph=I'$ and $\tilde g (0)=0$. Then, from the computations above, it is easy to see that the norms of $\tilde \ph$, $\tilde q$ and $\tilde g$ corresponding to $\tilde u$ are comparable to the ones of $\ph$, $q$ and $g$ respectively.
\end{proof}

\subsection{The linearized equation} We give below the proof of Lemma \ref{lineq} which is an application of the maximum principle and Taylor's formula. 
First we prove that
\begin{equation}\label{wxn}
|w| \le C x_n \quad \mbox{in} \quad B_{1/2}^+.
\end{equation}
Indeed, for example at $x_0=0\in \{x_n=0\} \cap B_{1/2}$, we consider the upper barrier
$$\bar w (x)= C x_n + 4 |x'|^2 -8n  x_n^{2+\alpha},$$
with $C$ large depending only on $n$ such that
$$ (C/2) x_n - 8n x_n^{2+\alpha}\geq 0~\text{on}~ B_1^{+}.$$
 Then $L(\bar w) \le 0$, and on $\p B_{1}^{+} \cap \{x_n>0\}$, we have
$\bar w \geq 1 \ge w.$
The maximum principle gives $w\leq \bar w$, and hence $w(0, x_n)\leq Cx_n$. We obtain the desired bound by translating the function $\bar w$ at other points on $\{x_n=0\} \cap B_{1/2}$.

Clearly $$\|w\|_{C^\frac{2}{2+\alpha}} \le C \|w\|_{L^\infty(B_1^+)} \quad \mbox{ in} \quad B=B_{1/4}(1/2e_n).$$ 
Now we apply this estimate for the rescalings
$w_h(x)=w(F_h x),$
use the bound \eqref{wxn}, and obtain
$$ \|w\|_{C^\frac{2}{2+\alpha}} \le C  \quad \mbox{ in} \quad F_h(B), \quad \forall h\in (0,1),$$
 which easily implies a bound on the norm of  $w$ in $C^\frac{2}{2+\alpha}( \ov B^+_{ 1/2})$.

The equation is invariant under translation in the $x_i$ direction for all $1\leq i\leq n-1$. Thus by iterating this argument (using first difference quotients) we find that the derivatives of $w$ with respect to  $x_i$ of any order  are 
bounded in $B_{1/2}^{+}$.

The equation $w_{nn}+x_n^\alpha \, \sum_{j=1}^{n-1}w_{jj}=0$  and the bound  $|w_{jj}| \leq C$  imply
the bound
$|w_{nn}| \le C\, x_n^\alpha.$
Thus, combining with the bound of $w_n$ on $\{x_n=0\}\cap B_{1/2}^{+}$, we see that $w_n$ is bounded in $B_{1/2}^{+}$. The same estimates as above show that
 $w_{in}$, $w_{iin}$ are bounded as well.

 By Taylor's formula, namely
$$f(t)=f(0)+f'(0) t+ \int_{0}^{t} (t-s) f''(s) ds$$
and the equation $Lw=0$,  we conclude that
$$w(0,x_n)=w(0)+w_n(0)\, x_n-\frac{\sum_{j=1}^{n-1}w_{jj}(0)}{(\alpha+2)(\alpha+1)}\, x_n^{2+\alpha}+O(x_n^{3+\alpha}),$$
$$w_i(0,x_n)=w_i(0)+w_{in}(0)\, x_n+O(x_n^{2+\alpha}), \quad \quad \quad i<n,$$
and
$$w_{ij}(0, x_n) = w_{ij}(0) + O(x_n), \quad  \quad \quad i,j<n.$$
Using these identities in the following Taylor's formula for $f(x^{'}) = w(x^{'}, x_n)$
\begin{eqnarray*}f(x^{'}) &=& f(0) + \sum_{i=1}^{n-1}f_i (0)x_i + \frac{1}{2}\sum_{i, j=1}^{n-1} f_{ij}(0)x_{i}x_j + O (|x^{'}|^3)\\
&=& w(0, x_n) + \sum_{i=1}^{n-1} w_{i}(0, x_n) x_i + \frac{1}{2}\sum_{i, j=1}^{n-1} w_{ij}(0, x_n)x_{i}x_j + O (|x^{'}|^3)
 \end{eqnarray*}
we obtain the desired conclusion since $w_{i}(0) = w_{ij}(0)=0$ for all $1\leq i, j\leq n-1$.
\qed

\subsection{Proof of Proposition \ref{LMA-model1}}After a dilation we may assume that $v$ solves 
\begin{equation}\label{veq2}
x_n^{\alpha} \, a^{ij} v_{ij} +  v_{nn}= x_n^{\alpha} \, f(x) \quad \mbox{in} \quad S_1(U_0),
\end{equation}
with $(a^{ij})_{i,j\le n-1}$ uniformly elliptic: $$\lambda I' \le (a^{ij}(x))_{i,j} \le \Lambda I',$$
and $a^{ij},f \in C^\gamma_\alpha(S_1(U_0))$ and $v(x',0)=\varphi(x') \in C^{2,\gamma}$.

As we remarked before, the distance $d_\alpha$ from Definition \ref{dal} occurs naturally in equation \eqref{veq2} since they both scale homogeneously after the transformations $x \mapsto F_h x$. Precisely,
$$\tilde v(x)=h^{-1}v(F_h x),$$
solves
$$x_n^{\alpha} \,\,  \tilde a^{ij} \tilde v_{ij} + \tilde v_{nn}= x_n^{\alpha} \tilde f(x),$$
with $$\tilde a^{ij}(x)=a^{ij}(F_h x), \quad \tilde f(x)=f(F_hx),$$
and
$$d_\alpha(y,z)=h^{- \frac 12} d_\alpha(F_hy,F_hz).$$

First we state a pointwise $C^{2,\gamma}$ estimate at the origin.
\begin{lemma}\label{po}
If $v$ satisfies (\ref{veq2}) as above, then there exists a tangent ``polynomial" $P_0$ at $0$,
$$P_0(x)=q(x')+(a_0+a'\cdot x') \, x_n + \frac{b_0}{(1+\alpha)(2+\alpha)} x_n^{2+\alpha},$$
with $q$ a quadratic polynomial in $x'$, and
$$a^{ij}(0) \, q_{ij}+b_0=f(0), \quad \quad \|q\|, |a'|, |a_0|, |b_0| \le C,$$
such that
$$|v-P| \le C U_0^{1+\frac{\gamma}{2}}\quad\text{in}\quad S_{1/2}(U_0).$$
The constant $C$ above depends on $n$, $\lambda$, $\Lambda$, $\alpha$, $\gamma$, $\|v\|_{L^\infty}$ and the corresponding norms of $a^{ij}$, $f$, $\varphi$.
\end{lemma}
Proposition \ref{LMA-model1} follows from Lemma \ref{po} by rescaling. Indeed, the function
$$w(x):=h^{-1-\frac{\gamma}{2}}(v-P_0) (F_hx) $$
satisfies 
$$ x_n^{\alpha}\, \, \tilde a^{ij} w_{ij} +  w_{nn}= x_n^{\alpha} \tilde f(x) \quad \mbox{in} \quad S_1(U_0)   $$
with
$$\tilde a^{ij}(x)=a^{ij}(F_hx), \quad \tilde f(x):=h^{-\frac \gamma 2} \left(f(F_hx)-f(0)-(a^{ij}(F_hx)-a^{ij}(0))q_{ij}  \right).  $$
Since $h^\frac 12 |y-z| \sim d_\alpha(F_hy,F_hz)$ for $y,z \in B_c(\frac 12 e_n)$, we see that
$$\|\tilde a^{ij}\|_{C^\gamma}, \|\tilde f\|_{C^\gamma} \le C \quad \mbox{in} \quad B:=B_{1/4}( e_n/2).$$
By Lemma \ref{po} $\|w\|_{L^\infty}\le C$, thus by the classical Schauder estimates we find
$$\|D^2 w\|_{C^\gamma} \le C \quad \mbox{in} \quad B.$$
Rescaling back and using that $D^2P \in C^\gamma_\alpha$, $D^2v(0)=D^2P(0)$, we find
$$\|D^2v -D^2v(0)\|_{L^\infty} \le C h^\frac \gamma 2, \quad \|D^2 v\|_{C^\gamma_\alpha} \le C \quad \mbox{in} \quad F_h (B),$$
and now it is straightforward to check the conclusion of Proposition \ref{LMA-model1}.   

\

Lemma \ref{po} is standard and we only sketch its proof below.
\begin{proof}[Proof of Lemma \ref{po}]
After subtracting the boundary data and a multiple of $x_n^{2+\alpha}$, and then after a suitable linear transformation, a dilation and multiplication by a constant we may assume that in $S_1(U_0)$, $v=0$ on $\{x_n=0\}$, $\|v\|_{L^\infty} \le 1$, and
$$a^{ij}(0)=\delta_{ij}, \quad [a^{ij}]_{C^\gamma_\alpha}\le \delta, \quad f(0)=0, \quad [f]_{C^\gamma_\alpha} \le \delta,$$
for some sufficiently small $\delta>0$. As before, we use barriers for $v$ near $x_n=0$ and conclude by compactness that, as $\delta \to 0$, $v$ can be approximated uniformly in $S_{1/2}(U_0)$ by a solution to
$$x_n^\alpha \Delta_{x'}\bar v +\bar v_{nn}=0, \quad \quad \bar v=0 \quad \mbox{on} \quad \{x_n=0\},\quad\quad
\|\bar v\|_{L^\infty} \le 1, $$
i.e. $$v=\bar v +o(1) \quad \mbox{in} \quad S_{1/2}(U_0) \quad \mbox{where}\quad o(1) \to 0 \quad \mbox{ as} \quad \delta \to 0.$$
By Lemma \ref{lineq}, there exists $\theta_0$ small depending only on $n$, $\alpha$ such that
$$|\bar v-\bar P| \le C_0 \theta_0^{\frac{3+\alpha}{2+\alpha}} \leq\frac{1}{2}\, \, \theta_0^{1+\frac \gamma 2} \quad \mbox{in} \quad S_{\theta_0}(U_0)$$
for some quadratic polynomial
$$\bar P (x):= (\bar a_0  + \bar a^{'}\cdot x') x_n.$$
In conclusion $$|v-\bar P| \leq \, \theta_0^{1+\frac \gamma 2} \quad \mbox{in} \quad S_{\theta_0}(U_0),$$
provided that $\delta$ is chosen sufficiently small. 

Now the rescaling of $v-\bar P$ from $S_{\theta_0}(U_0)$ to $S_1(U_0)$ satisfies the same hypotheses as $v$ above. Then we can iterate the same argument and obtain the desired conclusion.
\end{proof}

\end{document}